\def\N{\mathbb N}
\def\Z{\mathbb Z}

\def\R{\mathbb R}

\def\A{\mathcal A}

\def\pfz{\begin{proof}}
\def\pfk{\end{proof}}

\def\Zmb{{\mathbb Z}_{-\beta}}
\def\dl{d_{-\beta}(\ell)}

\newcommand{\ZB}{\mathbb Z_{\beta}}
\newcommand{\ZBP}{\mathbb Z_{\beta}^+}
\newcommand{\ZMB}{\mathbb Z_{-\beta}}

\newcommand{\TB}{T_{\beta}}
\newcommand{\TMB}{T_{-\beta}}
\newcommand{\DMBL}{d_{-\beta}(\ell)}
\newcommand{\DMBRS}{d_{-\beta}^*(\ell+1)}
\newcommand{\DBU}{d_\beta(1)}
\newcommand{\DBUS}{d_\beta^*(1)}

\newcommand{\rozvojp}[1]{\langle #1\rangle_{\beta}}
\newcommand{\rozvojm}[1]{\langle #1\rangle_{-\beta}}

\newcommand{\lprec}{\prec_{\mathrm{lex}}}
\newcommand{\lpreceq}{\preceq_{\mathrm{lex}}}

\newcommand{\aprec}{\prec_{\mathrm{alt}}}
\newcommand{\apreceq}{\preceq_{\mathrm{alt}}}

\documentclass[a4paper,11pt]{article}
\usepackage[cp1250]{inputenc}
\usepackage[IL2]{fontenc}
\usepackage[czech,english]{babel}
\usepackage{amssymb,amsmath,amsfonts,amsthm}
\usepackage{url}
\usepackage[dvips]{graphicx}

\newtheorem{lem}{Lemma}
\newtheorem{thm}[lem]{Theorem}
\newtheorem{prop}[lem]{Proposition}

\newtheorem{de}[lem]{Definition}
\newtheorem{ex}[lem]{Example}
\newtheorem{pozn}[lem]{Remark}

\begin{document}


\title{Confluent Parry numbers, their spectra, and integers in positive- and negative-base number systems}

\author{D. Dombek$^1$, Z. Mas\'akov\'a$^2$, T. V\'avra$^2$\\[2mm]
\small $^1$ Department of Applied Mathematics FIT, Czech Technical University in Prague\\
\small Th\'akurova 9, 160 00 Praha 6, Czech Republic\\[2mm]
$^2$ \small Department of Mathematics FNSPE, Czech Technical University in Prague\\
\small Trojanova 13, 120 00 Praha 2, Czech Republic}

\date{}

\maketitle

\begin{abstract}
In this paper we study the expansions of real numbers in positive
and negative real base as introduced by R\'enyi, and Ito \& Sadahiro, respectively. 
In particular, we compare the sets $\ZBP$ and $\ZMB$ of nonnegative $\beta$-integers and $(-\beta)$-integers. 
We describe all bases $(\pm\beta)$ for which $\ZBP$ and $\ZMB$ can be coded by
infinite words which are fixed points of conjugated morphisms, and consequently have the same language. 
Moreover, we prove that this happens precisely for $\beta$ with another interesting property, namely that 
any integer linear combination of non-negative powers of the base $-\beta$
with coefficients in $\{0,1,\dots,\lfloor\beta\rfloor\}$ is a
$(-\beta)$-integer, although the corresponding sequence of digits is forbidden as a $(-\beta)$-integer.
\end{abstract}

\section{Introduction}\label{sec_1intro}


In positive base number systems (defined by R\'enyi~\cite{Ren57}) many properties are specific for
the class of confluent Parry numbers, sometimes called generalized
multinacci numbers, i.e. zeros $>1$ of polynomials
\begin{equation}\label{eq:confluentpos}
x^d-mx^{d-1}-mx^{d-2}-\cdots-mx-n\,, \quad \text{where }\ d\geq 1,\ m\geq n\geq 1\,.
\end{equation}
Note that all such numbers are Pisot numbers, i.e.\ algebraic integers $>1$ with conjugates inside the unit circle.
Among exceptional properties of such bases is the existence of optimal
representations~\cite{DDKL}, or the fact that the infinite word
$u_\beta$ coding the set $\Z_\beta$ of $\beta$-integers is reversal closed and the
corresponding Rauzy fractal is centrally symmetric~\cite{BernatConfluent}.
Linear numeration systems for confluent Parry numbers are also mentioned in~\cite{Edson}, in connection 
to calculating the Garsia entropy.

From our point of view, the most important aspect is that any
integer linear combination of non-negative powers of the base
with coefficients in $\{0,1,\dots,\lceil\beta\rceil-1\}$ is a
$\beta$-integer, although the sequence of coefficients may be
forbidden in the corresponding R\'enyi number system, as shown
in 1992 by Frougny~\cite{Fr92}. In other words, normalization of
any representation using only non-negative powers of the base $\beta$ and the alphabet
$\{0,1,\dots,\lceil\beta\rceil-1\}$ into the greedy $\beta$-expansion does not
produce a fractional part. Formally, defining
\begin{equation}\label{eq:spectrum}
X(\beta) = \Big\{\sum_{j=0}^N a_j \beta^j : N\in\N,\, a_j\in\Z,\, 0\leq a_j<\beta\Big\}\,,
\end{equation}
one can write for confluent Parry numbers that $X(\beta)=\Z_\beta^+$.
Recall that the study of sets $X(\beta)$ for $\beta\in(1,2)$ was initiated by
Erd\H{o}s et al.~\cite{Erdos}. Later, the problem was generalized to considering coefficients $a_j$ in a more general alphabet $\{0,1,\dots,r\}$. The corresponding set is called the spectrum of $\beta$.

Our aim is to obtain a result analogous to that of~\cite{Fr92} in
case of negative base systems considered by Ito and
Sadahiro~\cite{IS09}. The question can be stated as follows: For
which $\beta>1$, any linear combinations of non-negative powers of
$-\beta$ with coefficients in $\{0,1,\dots,\lfloor\beta\rfloor\}$
is a $(-\beta)$-integer? Such a question was already put forth
in~\cite{MV13}, where it is shown that among quadratic numbers,
only zeros of $x^2-mx-m$, $m\geq 1$, have the required property.
Moreover, these are exactly the bases for which the distances between consecutive points in $\Z_\beta$
and $\Zmb$ take the same values $\leq 1$. Moreover, the
infinite words $u_{\beta}$,  $u_{-\beta}$ coding the ordering of distances in
$\Z_\beta$, resp.\ $\Z_{-\beta}$ have the same language, which follows from the fact that they
are fixed points of conjugated morphisms.

In this paper we solve the question in general. Defining
\begin{equation}\label{eq:spectrumnegative}
X(-\beta) = \Big\{\sum_{j=0}^N a_j (-\beta)^j : N\in\N,\, a_j\in\Z,\, 0\leq a_j\leq\beta\Big\}\,,
\end{equation}
we show that $X(-\beta)=\Z_{-\beta}$ if and
only if $\beta$ is a zero of~\eqref{eq:confluentpos} with the additional requirement that $m=n$ if $d$ is even.
In view of the result in the quadratic case, there is another natural question to ask:
For which $\beta>1$, the sets $\ZBP$ and $\ZMB$ are in some sense similar?
The similarity can be expressed by comparing the values of distances between consecutive points in $\ZBP$ and in $\Zmb$. Moreover,
one can study their ordering in the real line coded by infinite words $u_\beta$ and $u_{-\beta}$.
It turns out that the phenomena satisfied equivalently in the quadratic case are not equivalent for bases of higher degree.
However, the property $X(-\beta)=\Z_{-\beta}$ characterizes the bases, for which the infinite words $u_\beta$ and $u_{-\beta}$ are fixed by conjugated morphisms.
In particular, we prove the following theorem.

\begin{thm}\label{thm_hlavni_superteorem}
Let $\beta>1$.  Denote by $\varphi$ the canonical morphism of $\beta$ (over a finite or an infinite alphabet) and $\psi$ the antimorphism fixing the infinite word coding $\ZMB$.
Then the following conditions are equivalent:
\begin{enumerate}

\item $\beta$ is a zero of  $x^d-mx^{d-1}-\cdots -mx-n$ with $m\geq n\geq 1$, such that $n=m$
for $d$ even.

\item $\varphi^2$ is conjugated with $\psi^2$.

\item $\ZMB=X(-\beta)$.
\end{enumerate}
\end{thm}

%
%
%
%

The paper is organized as follows. In Section~\ref{sec2_prelim} we recall the definition and properties of numeration systems with positive and negative base, as defined by R\'enyi and Ito \& Sadahiro, respectively. In particular, we focus on the set of $\beta$-  and $(-\beta)$-integers and their coding by infinite words invariant under morphisms. In Section~\ref{sec_confluent} we consider $\beta>1$, the root of~\eqref{eq:confluentpos} and reproduce the proof that $X(\beta)=\Z_\beta^+$. Then we demonstrate implications $(1)\Rightarrow(2)$ and $(1)\Rightarrow(3)$ of Theorem~\ref{thm_hlavni_superteorem}. Implications $(2)\Rightarrow(1)$ and $(3)\Rightarrow(1)$ are given separately in Sections~\ref{sec_dan} and~\ref{sec_tomzuzana}. Let us mention that the equivalence $(1)\Leftrightarrow(2)$ was announced in a conference contribution~\cite{Dom13}.
Finally, Section~\ref{sec_5examples} contains some remarks.


\section{Preliminaries}\label{sec2_prelim}

Given $\alpha\in\R$, $|\alpha|>1$, every $x\in\R$ can be expressed as series $x=\sum_{i=-\infty}^N a_i\alpha^i$, $a_i\in\Z.$ Such a series is called an $\alpha$-representation of $x$ and is usually denoted by
$x = a_Na_{N-1}\cdots a_1a_0\bullet a_{-1}a_{-2}\cdots,$ where $\bullet$ separates between the coefficients at negative and non-negative powers of the base. An $\alpha$-representation is not unique. A specific type of $\alpha$-representations are the $\beta$-expansions and $(-\beta)$-expansions, whose properties we summarize below.

We will also need the notion of finite and infinite words over a given alphabet $\A$. We denote by $\A^*$ the monoid of finite words equipped with the operation of concatenation and the empty word $\epsilon$. A morphism over $\A$ is a mapping $\varphi:\A^*\to\A^*$ which satisfies $\sigma(vw)=\sigma(v)\sigma(w)$ for every pair of finite words $v,w\in\A^*$. An antimorphism $\psi:\A^*\to\A^*$ satisfies $\sigma(vw)=\sigma(w)\sigma(v)$ for $v,w\in\A^*$. The action of morphisms and antimorphisms can be extended to infinite words, in particular,
if $u= \cdots u_{-2}u_{-1}|u_0u_1u_2\cdots \in \A^\Z$ is a pointed bidirectional infinite word, then 
$$
\begin{aligned}
\varphi(u) &= \cdots \varphi(u_{-2})\varphi(u_{-1})|\varphi(u_0)\varphi(u_1)\varphi(u_{2})\cdots\\
\psi(u) &= \cdots \psi(u_{2})\psi(u_{1})\psi(u_0)|\psi(u_{-1})\psi(u_{-2})\cdots\\
\end{aligned}
$$


\subsection{R\'enyi $\mathbf{\beta}$-expansions}\label{sec_22renyi}

In 1957, R\'enyi~\cite{Ren57} defined the positional numeration
system with a positive real base. For $\beta>1,$ any
$x\in[0,1)$ has an expansion of the form
$d_\beta(x)=x_1x_2x_3\cdots$ defined by
\[
x_i=\lfloor\beta\TB^{i-1}(x)\rfloor,\text{ where }\TB(x)=\beta x-\lfloor\beta x\rfloor\,.
\]
Any $x\in[0,1)$ is then represented by an infinite string of
non-negative integer digits $<\beta$, formally, by an element of
$\A^\N=\{0,1,\ldots,\lceil\beta\rceil-1\}^\N$. Not every infinite
word over  $\A^\N$ does play the role of $d_\beta(x)$ of some
$x\in[0,1)$. Those that do are called admissible (or
$\beta$-admissible) and their characterization is due to
Parry~\cite{Par60}. He proved that a digit string
$x_1x_2\cdots\in\A^\N$ is admissible if and only if it fulfills
the lexicographic condition
\begin{equation}\label{eq_admissibilita_renyi}
0^\omega\lpreceq x_ix_{i+1}x_{i+2}\cdots\lprec\DBUS:=\lim_{y\to1_-}d_\beta(y)\text{ for all }i\geq 1\,.
\end{equation}
Here, we write $u^\omega=uuu\cdots$ for infinite repetition of the word $u$ and $\lprec$ stands for standard lexicographic ordering.
The limit is taken over the product topology on $\A^\N$. Let us point
out that the lexicographic ordering on admissible strings
corresponds to the ordering on the unit interval $[0,1)$, i.e.
$x<y$ if and only if $d_\beta(x)\lprec d_\beta(y)$.

Recall that the so-called R\'enyi expansion of unity is defined as
\begin{equation}\label{eq:db1}
\DBU=d_1d_2d_3\cdots,\text{ where $d_1=\lfloor\beta\rfloor$ and
$d_2d_3\cdots=d_\beta(\beta-\lfloor\beta\rfloor)$}\,.
\end{equation}
If $\DBU$ is eventually periodic, then $\beta$ is called a Parry
number. If, moreover, $\DBU=d_1\cdots d_k0^\omega$, then it is called a simple Parry number.

For the infinite R\'enyi expansion of unity $\DBUS$ one has
\[
\DBUS=\begin{cases}
(d_1\cdots d_{k-1}(d_k-1))^\omega & \text{if $\DBU=d_1\cdots d_k0^\omega$ with $d_k\neq 0$,}\\
\DBU & \text{otherwise.}
\end{cases}
\]

The notion of $\beta$-expansions can be naturally extended from $[0,1)$ to all reals.

\begin{de}\label{de_betarozvoje}
Let $\beta>1$, $x\in\R^+$. Let $k\in\N$ be minimal such that
$\frac{x}{\beta^k}\in[0,1)$ and
$d_\beta\Big(\frac{x}{\beta^k}\Big)=x_1x_2x_3\cdots$. Then the
$\beta$-expansion of $x$ is defined as
\[
\rozvojp{x}=\begin{cases}
x_1\cdots x_{k-1}x_k\bullet x_{k+1}x_{k+2}\cdots & \text{ if $k\geq 1$,}\\
0\bullet x_1x_2x_3\cdots & \text{ if $k=0$.}
\end{cases}
\]
For negative real numbers $x$ we use the
notation $-\rozvojp{|x|}$.
\end{de}

When $\beta$ is an integer, then the set of numbers whose
expansion uses only non-negative powers of $\beta$ is precisely
equal to $\Z$. As a natural generalization of $\Z$, the set $\ZB$
of $\beta$-integers can be defined for every $\beta>1$ using the
notion of $\rozvojp{x}$.

\begin{de}\label{de_betacela}
Let $\beta>1$. We define $\ZB=\ZBP\cup(-\ZBP)$, where
\[
\ZBP=\{x\geq0\ :\ \rozvojp{x} = x_k\cdots x_1x_0\bullet 0^\omega\}\ =\ \bigcup_{i\geq 0}\beta^i\TB^{-i}(0)\,.
\]
\end{de}

The distances between consecutive elements of $\ZB$ take values
\begin{equation}\label{eq_mezery_thurston}
\Delta_k=\sum_{i\geq 1}\frac{d'_{i+k}}{\beta^i},\quad
k=0,1,2,\ldots\,,
\end{equation}
where $\DBUS=d'_1d'_2d'_3\cdots$, see~\cite{Thu89}. Since
$\Delta_0=\sum_{i\geq 1}\frac{d'_i}{\beta^i}=1$ and any suffix of
$\DBUS$ either fulfills~\eqref{eq_admissibilita_renyi} or is equal
to $\DBUS$ itself, we get $\Delta_k\leq 1$ for all $k$.

We can encode the ordering of distances in
$\ZBP=\{z_0=0<z_1<z_2<\cdots\}$ by an infinite word
$v_\beta=v_0v_1v_2\cdots$ over the infinite alphabet $\N$. We set
$v_j=k$ if $k$ is the greatest index, at which the
$\beta$-expansions $\rozvojp{z_j}$ and $\rozvojp{z_{j+1}}$ differ.
Note that one has $z_{j+1}-z_j=\Delta_k$.

It can be seen that $v_\beta$ is a fixed point of a morphism
$\varphi:\N^*\rightarrow\N^*$, defined by
\begin{equation}\label{eq:fabreplus_obecne}
\varphi(i)=0^{d'_{i+1}}(i+1)\ \text{ for all }\ i\in\N\,.
\end{equation}

If $\beta$ is a Parry number, it is obvious
from~\eqref{eq_mezery_thurston} that the distances between
consecutive elements of $\ZBP$ take only finitely many values.
Then both $v_\beta$ and $\varphi$ can be projected onto a finite
alphabet. In particular, if $\DBU=d_1\cdots d_k 0^\omega$,
$d_k\neq 0$, then
\[
\Delta_{j+k}=\Delta_{j}\text{ for all $j\geq 0$}
\]
and the infinite word $u_\beta$ on the restricted alphabet $\{0,\dots,k-1\}$
is a fixed point of the morphism
\begin{equation}\label{eq:substitucesimpleParry}
\begin{aligned}
\varphi(i)=&\ 0^{d_{i+1}}(i+1)\ \text{ for } i\leq k-2\,,\\
\varphi(k-1)=&\ 0^{d_k}\,.
\end{aligned}
\end{equation}
If $\DBU=\DBUS=d_1\cdots d_k (d_{k+1}\cdots d_{k+p})^\omega$ with $k,p$ minimal, then
\[
\Delta_{j+p}=\Delta_{j}\text{ for all $j\geq k$},
\]
and the infinite word $u_\beta$ on the restricted alphabet
$\{0,\dots,k+p-1\}$ is a fixed point of the morphism
\begin{equation}\label{eq:substitucenonsimpleParry}
\begin{aligned}
\varphi(i)=&\ 0^{d_{i+1}}(i+1)\ \text{ for } i\leq k+p-2\,,\\
\varphi(k+p-1)=&\ 0^{d_{k+p}}k\,.
\end{aligned}
\end{equation}
These, the so-called canonical morphisms, were given
in~\cite{Fab95}.


\subsection{Ito-Sadahiro $(-\beta)$-expansions}\label{sec_22is}

In 2009, Ito and Sadahiro~\cite{IS09} considered a numeration
system with a base $-\beta<-1$. Any $x\in[\ell,\ell+1),$ where
$\ell=\frac{-\beta}{\beta+1},$ has an
 expansion of the form $d_{-\beta}(x)=x_1x_2x_3\cdots\in\A^\N$
defined by
\[
x_i=\lfloor-\beta\TMB^{i-1}(x)-\ell\rfloor\in\A=\{0,1,\ldots,\lfloor\beta\rfloor\}^\N,\text{
where }\TMB(x)=-\beta x-\lfloor-\beta x-\ell\rfloor\,.
\]

Ito
and Sadahiro proved that a digit string $x_1x_2x_3\cdots\in\A^\N$
is $(-\beta)$-admissible, i.e.\ equal to $d_{-\beta}(x)$ for some $x\in[\ell,\ell+1)$  if and only if
 it fulfills
\begin{equation}\label{eq_admissibilita_is} \DMBL\apreceq
x_ix_{i+1}x_{i+2}\cdots\aprec\DMBRS:=\lim_{y\to
{l+1}_-}d_{-\beta}(y)\text{ for all }i\geq 1\,.
\end{equation}
Here, $\aprec$ stands for alternate lexicographic ordering defined
as follows:
\[
x_1x_2\cdots\aprec y_1y_2\cdots\ \Leftrightarrow\
(-1)^k(x_k-y_k)<0\text{ for $k$ smallest such that }x_k\neq
y_k\,.
\]
The alternate ordering corresponds to the ordering on reals, i.e.\
$x<y$ if and only if $d_{-\beta}(x)\aprec d_{-\beta}(y)$.
Note that $\aprec$ can also be
used to compare finite digit strings when suffix $0^\omega$
is added to them.

In~\cite{IS09} it is also shown that
\[
\DMBRS=\begin{cases}
(0l_1\cdots l_{q-1}(l_q-1))^\omega & \text{if $\DMBL=(l_1l_2\cdots l_q)^\omega$ for $q$ odd,} \\
0\DMBL & \text{otherwise.}
\end{cases}
\]

Now the expansion can be defined for all reals without the need of a minus sign.
\begin{de}\label{de_minusbetarozvoje}
Let $-\beta<-1$, $x\in\R$. Let $k\in\N$ be minimal such that
$\frac{x}{(-\beta)^k}\in(\ell,\ell+1)$ and
$d_{-\beta}\Big(\frac{x}{(-\beta)^k}\Big)=x_1x_2x_3\cdots$. Then
the $(-\beta)$-expansion of $x$ is defined as
\[
\rozvojm{x}=\begin{cases}
x_1\cdots x_{k-1}x_k\bullet x_{k+1}x_{k+2}\cdots & \text{ if $k\geq 1$,}\\
0\bullet x_1x_2x_3\cdots & \text{ if $k=0$.}
\end{cases}
\]
\end{de}

Similarly as in a positive base numeration, the set of
$(-\beta)$-integers $\ZMB$ can now be defined using the notion of
$\rozvojm{x}$. Note that $\ZMB$
coincides with $\Z$ if and only if  $\beta$ is an integer.

\begin{de}\label{de_minusbetacela}
Let $-\beta<-1$. Then the set of $(-\beta)$-integers is defined as
\[
\ZMB=\{x\in\R\ :\ \rozvojm{x} = x_k\cdots x_1x_0\bullet
0^\omega\}\ =\ \bigcup_{i\geq 0}(-\beta)^i\TMB^{-i}(0)\,.
\]
\end{de}

It can be shown that if $1<\beta<\tau=\frac12(1+\sqrt5)$, then the set of $(-\beta)$-integers is trivial, $\ZMB=\{0\}$.
It is therefore reasonable to consider $\beta\geq\tau$. In order to describe the distances between consecutive
$(-\beta)$-integers, we will recall some notation
from~\cite{ADMP12}. Let
\begin{align*}
\min(k)= & \min\{a_{k-1}\cdots a_1a_0\ :\ a_{k-1}\cdots a_1a_0
0^\omega\text{ is admissible}\}\,,
\end{align*}
where $\min$ is taken with respect to the alternate order on
finite strings. Similarly we define $\max(k)$. Furthermore, let
$\gamma$ be the ``value function'' mapping finite digit strings to
real numbers,
\[
\gamma:\quad x_{k-1}\cdots x_1 x_0 \quad\mapsto\quad
\gamma(x_{k-1}\cdots x_1 x_0)=\sum_{i=0}^{k-1}x_i(-\beta)^i\,.
\]
 It was shown in~\cite{ADMP12} that the distances between
consecutive elements $x<y$ of $\ZMB$ take values
$y-x\in\{\Delta'_k$\ :\ $k\in\N\},$
\begin{equation}\label{eq_vzdalenosti_zmb_predpis}
\Delta'_k=\Big|(-\beta)^k+\gamma\big(\min(k)\big)-\gamma\big(\max(k)\big)\Big|\,,
\end{equation}
where $k$ is the greatest index at which $\rozvojm{x}$ and
$\rozvojm{y}$ differ.

We can encode the ordering of distances in $\ZMB=\{\cdots <z_{-1}<z_0=0<z_1<\cdots\}$ by a biinfinite word
\[
v_{-\beta}=\cdots v_{-3}v_{-2}v_{-1}|v_0v_1v_2\cdots,\quad
v_i\in\{0,1,2,\ldots\}\,,
\]
where $v_j=k$ if $k$ is the greatest index, at which the $(-\beta)$-expansions of  $z_j$ and $z_{j+1}$ differ.
Note that $z_{j+1}-z_j=\Delta'_k$.

In~\cite{ADMP12} it is shown that
there exists an antimorphism $\psi:\N^*\rightarrow\N^*$ such that
$\psi^2$ is a non-erasing non-identical morphism and
$\psi(v_{-\beta})=v_{-\beta}$.
Moreover, $\psi$ is of the form
\begin{equation}\label{thm_antimorfismy}
\psi(k)=\left\{\begin{array}{ll}
S_{k}(k+1)\widetilde{R_{k}}&\quad\text{for $k$ even},\\[2mm]
R_{k}(k+1)\widetilde{S_{k}}&\quad\text{for $k$ odd}\,,
\end{array}\right.
\end{equation}
where $\widetilde{u}$ denotes the mirror image of the word $u$.
The word $S_k$ codes the distances between consecutive
$(-\beta)$-integers
$\{\gamma(\min(k)0),\ldots,\gamma(\min(k+1))\}$ (in given order)
and similarly $R_k$ in
$\{\gamma(\max(k)0),\ldots,\gamma(\max(k+1))\}$.

When $\DMBL$ is eventually periodic, then $\beta$ is called an Yrrap number (or Ito-Sadahiro number).
In this case, $v_{-\beta}$ and $\psi$ can
be projected to a finite alphabet as the distances of the same
length can be coded by the same letter, see~\cite{ADMP12,Ste12}.
The infinite word over the restricted alphabet is denoted $u_{-\beta}$.

\section{Confluent Parry numbers}\label{sec_confluent}

Let us study the properties of $\beta$- and $(-\beta)$-integers in case that $\beta>1$ is a confluent Parry number, i.e.\ a zero of $x^d-mx^{d-1}-mx^{d-2}-\cdots-mx-n$, where $d\geq 1,\ m\geq n\geq 1$. Their significance was first observed by Frougny in~\cite{Fr92} who shows that the corresponding linear numeration system is confluent. We present the formulation using the spectrum $X(\beta)$, for which we aim to find an analogue in case of negative base. For illustration of the differences between negative and positive base, we include its proof.

\begin{thm}\label{t:confluenceFrougny}
Let $\beta>1$. Then $X(\beta) 
= \Z_\beta^+$ if and only if $\beta$ is a zero of~\eqref{eq:confluentpos}.
\end{thm}

\pfz
Obviously, $X(\beta) \supset \Z_\beta^+$ for every $\beta>1$.
Realize that $\beta>1$ satisfying~\eqref{eq:confluentpos} is equivalent to
\begin{equation}\label{eq:confldb}
d_\beta(1)=m^{d-1}n0^\omega\,.
\end{equation}
First, let us show that~\eqref{eq:confldb} implies $X(\beta)\subset\Z_\beta^+$.
Let $x=\sum_{i=0}^N a_i\beta^i\in X(\beta)$. If $\langle x\rangle_\beta=a_N\cdots a_1a_0\bullet 0^\omega$, then obviously
$x\in \Z_\beta^+$. In the opposite case, by the Parry condition~\eqref{eq_admissibilita_renyi}, we derive that $a_N\cdots a_1a_0$
contains a substring which is lexicographically greater than $m^{d-1}n$, consequently, the string $0a_N\cdots a_1a_0$ contains a substring
$ym^{d-1}z$, where $0\leq y<m$ and $n<z\leq m$.
Since $\beta$ is a zero of~\eqref{eq:confluentpos}, replacing
the substring $ym^{d-1}z$ in $0a_N\cdots a_1a_0$ by the substring $(y+1)0^{d-1}(z-n)$,
  we find a representation of the same number $x$ with strictly smaller digit sum. Consequently, after a finite number of such steps, we obtain a representation $b_Kb_{K-1}\cdots b_1b_0\bullet0^\omega$ of $x$ which does not contain any substring lexicographically greater than $m^{d-1}n$, which by~\eqref{eq_admissibilita_renyi} shows that $\langle x\rangle_\beta=b_Kb_{K-1}\cdots b_1b_0\bullet0^\omega$, i.e. $x\in\Z_\beta^+$.

It remains to show that~\eqref{eq:confldb} is necessary for
$X(\beta) \subset \Z_\beta^+$. Let
$d_\beta(1)=d_1d_2d_3\cdots$ and let $i\geq 2$ be minimal, such
that $d_i<d_1$, i.e. $d_\beta(1)=d_1d_1\cdots d_1 d_i
d_{i+1}\cdots$. Suppose that $d_{i+1}d_{i+2}\cdots \neq 0^\omega$.
We find an element of the spectrum $X(\beta)$ which is not a $\beta$-integer.
Consider $z= d_i+1+\sum_{j=1}^{i-1
}d_1\beta^j= \underbrace{d_1 d_1
\cdots d_1}_{\text{\tiny $i-1$ times}}(d_i+1)\bullet$.
By~\eqref{eq_admissibilita_renyi}, we have $d_{i+1}d_{i+2}\cdots
\lprec d_1d_2\cdots$ and the lexicographic ordering corresponds to
the natural order in $[0,1]$. Consequently,
$$
z-\beta^i=1-\sum_{k=1}^{+\infty} \frac{d_{i+k}}{\beta^k}\in(0,1)\,.
$$
If $\beta^i<z<\beta^{i+1}$, then the proof is finished, because the $\beta$-expansion of $z$ is of the form $\langle z\rangle_\beta=10^i\bullet z_1z_2z_3\cdots$, where $\langle z-\beta^i\rangle_\beta=0\bullet z_1z_2z_3\cdots\neq 0^\omega$. Therefore $z\in X(\beta)$ but $z\notin\ZBP$.

It remains to solve the case that $z\geq \beta^{i+1}$. From $\beta^{i+1}\leq z<\beta^i+1$, we derive that $d_1=\lfloor\beta\rfloor=1$ and $\beta^{k}(\beta-1)=T^{k}(\beta-1)<1$ for $0\leq k\leq i$. This, by definition~\eqref{eq:db1} of $d_\beta(1)$ implies
$d_{k+1}=\lfloor\beta T^{k-1}(\beta-1)\rfloor=0$ for $1\leq k\leq i$. Hence we have $d_\beta(1)=10^j1d_{j+3}d_{j+4}\cdots$, where $j\geq i\geq 2$. Consider $w=\beta^{j}+1=10^{j-1}1\bullet$, i.e.
$$
w-\beta^{j+1}=1-\frac1{\beta}-\sum_{k=2}^{+\infty} \frac{d_{j+1+k}}{\beta^k}\in(0,1)\,.
$$
Since  $d_{j+2}=\lfloor\beta^{j+1}(\beta-1)\rfloor=1$, it follows that
$$
1\leq \beta^{j+2}-\beta^{j+1}<\beta^{j+2}-\beta^{j} = \beta^{j+2}-w+1\,,
$$
and therefore $w<\beta^{j+2}$. Necessarily, the $\beta$-expansion of $w$ is of the form $\langle w\rangle_\beta  = 10^{j+1}\bullet w_1w_2w_3\cdots$,
where $\langle w-\beta^{j+1}\rangle_\beta=0\bullet w_1w_2w_3\cdots\neq 0^\omega$. Therefore $w\in X(\beta)$ but $w\notin\ZBP$.
\pfk

From the above theorem, one can see that for the description of the gap sequence in the spectrum $X(\beta)$, it is sufficient to use the knowledge about $\beta$-integers. Since $d_\beta(1) =m^{d-1} n$, from~\eqref{eq_mezery_thurston} and~\eqref{eq:substitucesimpleParry}, we derive that the gaps in $X(\beta)=\Z_\beta$ take values
$$
\Delta_i=\,m\Big(\frac{1}{\beta}+\cdots+\frac{1}{\beta^{d-1-i}}\Big)+\frac{n}{\beta^{d-i}}\quad\text{for $0\leq i\leq d-1$}\,.
$$
and the gap sequence in $X(\beta)=\Z_\beta^+$ is coded by the infinite word $u_\beta$ over the alphabet $\{0,\dots,d-1\}$, which is a fixed point
of the morphism
$$
\varphi(i)= 0^{m}(i+1)\ \text{ for } i\leq k-2\,,\quad \varphi(d-1)= 0^{n}\,.
$$

Let us now study what role do play confluent Parry numbers $\beta$ in systems with negative base $-\beta$.
We have
$$
d_{-\beta}(\ell) =
\begin{cases}(m0)^{k-1}m(m-n)^\omega & \text{ if } d=2k,\\
(m0)^kn^\omega &  \text{ if } d=2k+1.
\end{cases}
$$

Let first $d$ be odd or $m=n$. Then it can be verified using~\eqref{eq_vzdalenosti_zmb_predpis} that we get the same set of distances in $\ZMB$ as in $\ZBP$, and, moreover, they are ``ordered'' the same way, i.e. $\Delta_i=\Delta'_i$, where $i$ always corresponds to the greatest index at which the expansions of two neighbors in $\ZBP$ or $\ZMB$ differ. The infinite word $u_{-\beta}$ coding $\ZMB$ is a fixed point of the antimorphism
$$
\psi(i)= 0^{m}(i+1)\ \text{ for } 0\leq i\leq d-2\,,\quad \psi(d-1)= 0^{n}\,.
$$
Although the prescriptions for the morphism $\varphi$ and antimorphism $\psi$ coincide, for comparing them, we have to use the second iteration,
and compare the morphisms $\varphi^2$, $\psi^2$. We have
\[\begin{array}{c|ll|ll}
i           & \varphi(i)    & \varphi^2(i)  & \psi(i)   & \psi^2(i)\\ \hline
0           & 0^m1          & (0^m1)^m0^m2  & 0^m1      & 0^m2(0^m1)^m\\
1           & 0^m2          & (0^m1)^m0^m3  & 0^m2      & 0^m3(0^m1)^m\\
\vdots  &      \vdots             & \vdots            &    \vdots           & \vdots\\
(d-2)       & 0^m(d-1)      & (0^m1)^m0^n   & 0^m(d-1)  & 0^n(0^m1)^m\\
(d-1)   & 0^n               & (0^m1)^n      & 0^n           & (0^m1)^n
\end{array}\]
wherefrom it can be seen that
$$
\varphi^2(i)(0^m1)^m=(0^m1)^m\psi^2(i)\,,\quad\text{ for all }i\in\{0,\ldots,d-1\}\,.
$$
This means that the morphisms $\varphi^2$, $\psi^2$ are conjugated.

\begin{de}\label{de_konjugace}
Let $\A$ be an alphabet (finite or infinite) and $\pi,\rho:\A^*\rightarrow\A^*$ be morphisms on $\A$. We say that $\pi$ and $\rho$ are conjugated if there exists a word $w\in\A^*$ such that either \[w\pi(a)=\rho(a)w,\text{ for all }a\in\A\,,\text{ or }\pi(a)w=w\rho(a),\text{ for all }a\in\A\,.\] We denote $\pi\sim\rho$.
\end{de}

It is well known that languages of fixed points of conjugated morphisms coincide. As a consequence, we have the following proposition, which is in fact implication $(1)\Rightarrow(2)$ of Theorem~\ref{thm_hlavni_superteorem}.

\begin{prop}\label{p:1->2}
Let $\beta>1$ be a zero of $x^d-mx^{d-1}-\cdots -mx-n$ with $m\geq n\geq 1$, such that $d$ is odd or $n=m$.
Then the infinite words $u_\beta$ and $u_{-\beta}$ have the same language.
\end{prop}

\begin{ex}\label{ex_peknatrida}
Consider now $\beta>1$ zero of $x^d-mx^{d-1}-\cdots-mx-n$, where $d$ is even and $1\leq n<m$.
Relation~\eqref{eq_vzdalenosti_zmb_predpis} implies that not all values of distances in $\ZMB$ correspond to their $\ZBP$ counterparts. In particular,
\[
\begin{aligned}
\Delta'_i &=\Delta_i\,,\quad \text{ for }i=0,\dots,d-2\,,\\
\Delta'_{d-1}&=1+\frac{n}{\beta}=\Delta_{d-1}+1>1.
\end{aligned}
\]
This implies that the morphisms $\varphi^2$ and $\psi^2$ cannot be connected by any similar property as in the previous case.
\end{ex}

\bigskip
In analogy with Theorem~\ref{t:confluenceFrougny}, we would like to compare the spectrum $X(-\beta)$ (defined by~\eqref{eq:spectrumnegative})
with the set of $(-\beta)$-integers. The question, however, is much more complicated in negative base.
We first describe the gaps in $X(-\beta)$ and give an antimorphism under which the gap sequence of $X(-\beta)$ is invariant (see Proposition~\ref{thm_erdos_set}). It will be seen that the gaps and the antimorphism coincide with those for $\ZMB$ when $\beta$ is a zero of~\eqref{eq:confluentpos} with $d$ odd or $m=n$. For such $\beta$, this proves that $\ZMB=X(-\beta)$.
which constitutes implication $(1)\Rightarrow(3)$ of Theorem~\ref{thm_hlavni_superteorem}. The fact that no other $\beta$ has this property (i.e.\ implication $(3)\Rightarrow(1)$) is more complicated and is demonstrated in Section~\ref{sec_tomzuzana}.

\begin{prop}\label{thm_erdos_set}
  Let $\beta>1$. The gaps $y_{j+1}-y_j$ in $X(-\beta)=\{\cdots < y_{-1}<0=y_0<y_1<\cdots\}$ are $\leq 1$. In particular, if $\beta$ is a zero of~\eqref{eq:confluentpos}, then the gaps take values
$$
\Delta_i=\,m\Big(\frac{1}{\beta}+\cdots+\frac{1}{\beta^{d-1-i}}\Big)+\frac{n}{\beta^{d-i}}\quad\text{for $0\leq i\leq d-1$}\,.
$$
  Moreover, the infinite word $u=\cdots u_{-2}u_{-1}|u_0u_1u_2\cdots$ coding the gap sequence of $X(-\beta)$ by $u_j=i$ if $y_{j+1}-y_j=\Delta_i$, is a fixed point of the antimorphism $\psi:\{0,1,\dots,d-1\}^*\to\{0,1,\dots,d-1\}^*$,
  $$
  \psi(i)=0^m(i+1),\text{ for }0\leq i\leq d-2,\qquad \psi(d-1)=0^n\,.
  $$
\end{prop}

\begin{proof}
  Let us define the sets
  \begin{align*}
    A_0&=\big\{0,1,\dots,\lfloor\beta\rfloor\big\},\\
    A_{n+1}&=(-\beta)A_{n}+\big\{0,1,\dots,\lfloor\beta\rfloor\big\}.
  \end{align*}
It can be easily verified that $X(-\beta)=\bigcup_{n\in\N}A_n$ and that the gaps between consecutive points of $A_n$ are $\leq1$ for any $\beta>1$. Since $A_{n}\subset\A_{n+1}$, it shows that gaps in $X(-\beta)$ are $\leq 1$.

Let $\beta>1$ be a zero of~\eqref{eq:confluentpos}. We will show by induction that gaps between consecutive points of $A_n$ take only values $\Delta_i$, and that the gap $\Delta_{d-1}$ always follows the gap $\Delta_0$. The idea of the proof is illustrated in Figure~\ref{f}.

\begin{figure}[ht]
\includegraphics[width=\textwidth]{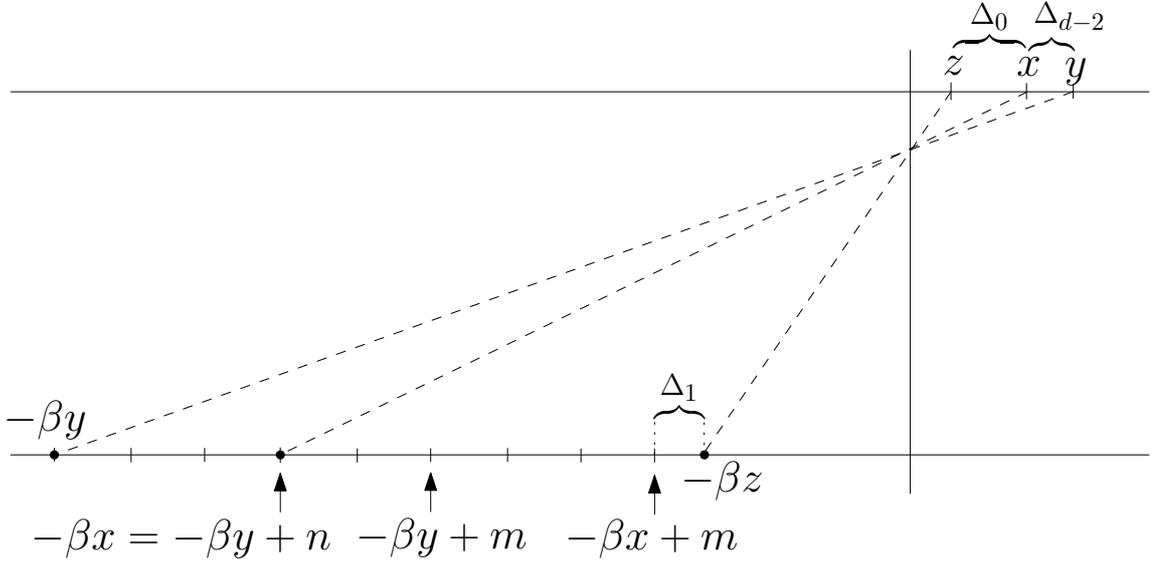}
\caption{Construction of the antimorphism $\psi$ for $u_{-\beta}$, where $\beta$ is the root of~\eqref{eq:confluentpos} with $m=5$, $n=3$.}
\label{f}
\end{figure}

Suppose that $x,y$ are consecutive points in $A_n$ such that $y=x+\Delta_i$ for some $0\leq i<d-1$. Then
\begin{equation}\label{eq:e1}
\{-\beta y\}+ \{0,\dots,m\} = \{-\beta y,-\beta y+1,\dots , -\beta y+m\} = A_{n+1}\cap [-\beta y,-\beta x)\,.
\end{equation}
  Note that all gaps $[-\beta y,-\beta x]\cap A_{n+1}$ are of length $\Delta_0=1$, except $(-\beta x)-(-\beta y+m)=\beta(y-x-\frac m\beta)=\beta(\Delta_i-\frac m\beta)=\Delta_{i+1}$. Moreover, $\Delta_{d-1}$ appears only if $y-x=\Delta_{d-2}$ and is preceded by $\Delta_0=1$.

Suppose now that $z<x<y$ where  $y-x=\Delta_{d-1}=\frac{n}{\beta}$ and $x-z=\Delta_0=1$. Note that $-\beta y+n = -\beta x$, therefore
we have
\begin{equation}\label{eq:e2}
\begin{aligned}
\{-\beta y,-\beta x\}+ \{0,\dots,m\} &= \{-\beta y,-\beta y+1,\dots,-\beta y+n=-\beta x,\dots -\beta x+m\} =\\
 &=A_{n+1}\cap [-\beta y,-\beta z)\,.
\end{aligned}
\end{equation}
Between $-\beta y$ and $-\beta x$ one obtains only gaps $\Delta_0$. Consequently, the gaps between consecutive points of $A_n$ are of the form $\Delta_i$ for every $n$.

From the above relations~\eqref{eq:e1} and~\eqref{eq:e2}, we can also read that $X(-\beta)$ is invariant under the antimorphism
$$
0\to 0^m 1,\ 1\to 0^{m}2,\ \dots,\ (d-2)\to 0^m(d-1),\ (d-1)\to 0^n\,.
$$
\end{proof}

\section{Proof $\boldsymbol{(2)\Rightarrow(1)}$ of Theorem~\ref{thm_hlavni_superteorem}}\label{sec_dan}

Our aim is to show that the fact that $\varphi^2\sim\psi^2$ implies that $\beta$ belongs to a specific class of numbers.
In what follows, we denote by $\varphi$ the canonical morphism of $\beta$. In particular, if $\beta$ is a Parry number, then $\varphi$ is given by~\eqref{eq:substitucesimpleParry} or~\eqref{eq:substitucenonsimpleParry}, and if $\beta$ is not a Parry number, then $\varphi$ is the
morphism over $\N$, given by~\eqref{eq:fabreplus_obecne}. The antimorphism $\psi$ figuring in $\varphi^2\sim\psi^2$ is the infinite
antimorphism~\eqref{thm_antimorfismy} fixing $v_{-\beta}$, if $\varphi$ is infinite; or its projection fixing the infinite word $u_{-\beta}$, otherwise.

\begin{lem}\label{lem_stejne_mezery}
Let $\beta>1$. Denote by $\Delta_i$ the distances in $\ZBP$ given by~\eqref{eq_mezery_thurston} and $\Delta_i'$ the distances in $\ZMB$ given by~\eqref{eq_vzdalenosti_zmb_predpis}. Let $\varphi$ be the canonical morphism of $\beta$ and $\psi$ an antimorphism fixing the infinite word coding $\ZMB$. If $\varphi^2\sim\psi^2$, then either both $\{\Delta_0,\Delta_1,\ldots\}$ and
$\{\Delta'_0,\Delta'_1,\ldots\}$ are infinite sets or they have the same cardinality. Moreover,
$\Delta_i=\Delta'_i$ for all $i$.
\end{lem}

\begin{proof}
We can assume the existence of at least three distinct distances in $\ZBP$ and $\ZMB$, since the case with one distance corresponds to integer bases and two distances correspond to quadratic bases, already solved in $\cite{MV13}$.

Thus $\varphi$ is defined over at least three letters. Denoting $m=\lfloor\beta\rfloor$, one has $\varphi(0)=0^m1$, $\varphi(1)=0^{d_2}2$, which implies $\varphi^2(0)=(0^m1)^m0^{d_2}2$. Since $\varphi^2\sim\psi^2$, the words $\varphi^2(0)$ and $\psi^2(0)$ both contain the same
number of zeros and ones and only one letter $2$ and we have
\begin{equation}\label{eq_mezery_z_obrazu_nuly}
\beta^2\Delta_0=a\Delta_0+b\Delta_1+\Delta_2=a\Delta'_0+b\Delta'_1+\Delta'_2\,,
\end{equation}
where $a=m^2+d_2$ and $b=m$. It holds that $\Delta_0=\Delta'_0=1$, and since $\min(1)\in\{m-1,m\}$, $m=\lfloor\beta\rfloor$, by~\eqref{eq_vzdalenosti_zmb_predpis}, we obtain $\Delta'_1\in\{\Delta_1,\Delta_1+1\}$.
Assuming $\Delta'_1=\Delta_1+1$ together with $\Delta_1<1 $ leads to $\Delta'_2\leq 0$, which is a contradiction.

Consequently, $\Delta_0=\Delta'_0$, $\Delta_1=\Delta'_1$ and $\eqref{eq_mezery_z_obrazu_nuly}$ imply $\Delta_2=\Delta'_2$. We can obtain the statement by repeating the same process for $\varphi^2(i)\sim\psi^2(i)$, $i\geq1$.
\end{proof}

In the sequel, we use the following statement from combinatorics on words; its proof can be found in~\cite{lothaire1}. It will be useful for determining the word $w$ for conjugation of morphisms as in Definition~\ref{de_konjugace}.

\begin{prop}\label{p:lothaire}
Let $x,y,w\in\A^*$ satisfy $wx = yw.$ Then $w$ is a prefix of $y^\omega.$
\end{prop}

It follows that when morphisms $\varphi$ and $\psi$ are conjugated, i.e.\ there exists $w$ such that $w\varphi(a)=\psi(a) w$ for all $a\in\A,$ or $\varphi(a)w = w\psi(a)$ for all $a\in\A,$ then $w$ is a common prefix of all $(\psi(a))^\omega$ or $(\varphi(a))^\omega$ respectively.

\begin{prop}\label{prop_konjugace_implikuje_parry}
Let $\varphi^2\sim\psi^2$. Then $\beta$ is a Parry number.
\end{prop}

\begin{proof}
If the morphisms $\varphi^2$ and $\psi^2$ are conjugated, then the distances between consecutive elements in $\ZBP$ and in $\ZMB$ coincide. This in turn implies that $\Delta_i'\leq 1$ for all $i$. From $\Delta'_1=\beta-\min(1)$ we obtain that $\min(1)=m=\lfloor\beta\rfloor$, hence the string $m0^\omega$ is \hbox{$(-\beta)$-admissible}. Necessarily, either $\DMBL=m0^\omega$ and thus $\beta$ is a quadratic number, zero of $x^2-mx-m$, or
$$
\DMBL=m0^{2k-1}a\cdots,\ a,k\geq1\,.
$$
In the latter case we have $\min(0)=\epsilon=\max(0)$, $\min(1)=m$, $\max(1)=0$, $\min(2)=m0$, $\max(2)=0m$, which, by~\eqref{thm_antimorfismy}, gives
$\psi(0)=0^m1$, $\psi(1)=0^m2$. Consequently, $\psi^2(0)=0^m2(0^m1)^m$. From $\varphi^2\sim\psi^2$, we derive that that the word $\varphi^2(0)=(0^m1)^m0^{d_2}2$ has the same number of occurrences of $0$ as the word $\psi^2(0)$, i.e. $d_2=m$.
Since now $\beta$ is not quadratic, we derive that
\begin{equation}\label{eq:mm}
\DBUS=mm\cdots\,.
\end{equation}

Now assume that $\beta>1$ is not a Parry number, hence $\DBUS$ is
aperiodic. We compare the morphisms $\varphi^2$ and $\psi^2$,
\[
\begin{array}{c|ll|ll}
i & \varphi(i) & \varphi^2(i) & \psi(i) & \psi^2(i)\\ \hline
0           & 0^m1                  & (0^m1)^m0^{m}2                          & 0^m1  & 0^m2(0^m1)^m\\
1           & 0^{m}2                  & (0^m1)^{m}0^{d_3}3                      & 0^m2  & \psi(2)(0^m1)^m\\
\vdots  & \vdots                    & \vdots                                    & \vdots    & \vdots\\
k           & 0^{d_{k+1}}(k+1)  & (0^m1)^{d_{k+1}}0^{d_{k+2}}(k+2)  & \psi(k)   & \psi^2(k)\\
\vdots  & \vdots                    & \vdots                                    & \vdots    & \vdots
\end{array}
\]
If $\varphi^2\sim\psi^2$, then there exists a word $w$ such that $\varphi^2(i)w=w\psi^2(i)$ for all $i$ or $w\varphi^2(i)=\psi^2(i)w$ for all $i$. Assume that $w\varphi^2(i)=\psi^2(i)w$. From Proposition~\ref{p:lothaire}, $w$ is a prefix of $\big(\psi^2(0)\big)^\omega= \big(0^m2(0^m1)^m\big)^\omega$. This is not possible, since from $w\varphi^2(1)=\psi^2(1)w$, the last letter of $w$ is $3$. Therefore necessarily $\varphi^2(i)w=w\psi^2(i)$ for all $i$, and $w$ is a prefix of $\big(\varphi^2(0)\big)^\omega= \big((0^m1)^m0^m2\big)^\omega$, moreover, having suffix $(0^m1)^m$. 

Comparing with $\varphi^2(1)w=w\psi^2(1)$, we obtain $w=(0^m1)^m$, which therefore must be a prefix of all $\varphi^2(i)$ for $i\geq 0$. This implies $d_{i+1}=m$ for $i\geq 0$, i.e. $\DBU=m^\omega$, which gives $\beta\in\N$. This shows that $\DBU$ cannot be aperiodic.
\end{proof}

\begin{prop}\label{prop_konjugace_implikuje_simpleparry}
Let $\varphi^2\sim\psi^2$. Then $\beta$ is a simple Parry number.
\end{prop}

\begin{proof}
Assume that $\beta>1$ is a non-simple Parry number, i.e. let
\[
\DBUS=d_1\cdots d_k(d_{k+1}\cdots d_{k+p})^\omega
\]
for $k,p\geq 1$ chosen minimal. As we have shown (cf.~\eqref{eq:mm}), $\varphi^2\sim\psi^2$ implies
that $\DBUS$ (not purely periodic) has prefix $mm$, hence $k+p\geq 3$.

We distinguish several subcases. At first, if $k+p=3$,
then either $d_\beta^*(1)=m(md_3)^\omega$ (case A) or
$d_\beta^*(1)=mm(d_3)^\omega$ (case B). Hence
\[
\begin{array}{c|ll|ll|ll} i & \varphi(i)\ (A) & \varphi^2(i)\
(A)     & \varphi(i)\ (B) & \varphi^2(i)\ (B)       & \psi(i)   &
\psi^2(i)\\ \hline
0 & 0^m1                & (0^m1)^m0^m2      & 0^m1              & (0^m1)^m0^m2          & 0^m1      & 0^m2(0^m1)^m\\
1 & 0^m2                & (0^m1)^m0^{d_3}1  & 0^m2              & (0^m1)^m0^{d_3}2      & 0^m2      & \psi(2)(0^m1)^m\\
2 & 0^{d_3}1            & (0^m1)^{d_3}0^m2  & 0^{d_3}2          & (0^m1)^{d_3}0^{d_3}2  & \psi(2)   & \psi^2(2)
\end{array}
\]
In both cases, we can use Proposition~\ref{p:lothaire} and similarly as in the proof of Proposition~\ref{prop_konjugace_implikuje_parry}, we derive that $\varphi^2(i)w=w\psi^2(i)$ where $w=(0^m1)^m$. Necessarily $d_3=m$, which is a contradiction with $\DBUS$ not being
purely periodic.

Assume that $k+p\geq 4$. Then there are at least four distinct
gaps $\Delta_{0,1,2,3}$ in $\ZBP$ and $\ZMB$, and we have
\begin{equation}\label{eq_tabulka_nonsimple_Parry}
\begin{array}{c|ll|ll}
i           & \varphi(i)                & \varphi^2(i)                                                  & \psi(i)       & \psi^2(i)\\ \hline
0           & 0^m1                      & (0^m1)^m0^m2                                              & 0^m1          & 0^m2(0^m1)^m\\
1           & 0^m2                      & (0^m1)^m0^{d_3}3                                          & 0^m2          & \psi(2)(0^m1)^m\\
\vdots  & \vdots                        & \vdots                                                            & \vdots            & \vdots\\
(k\!+\!p\!-\!3) & 0^{d_{k+p-2}}(k\!+\!p\!-\!2)  & (0^m1)^{d_{k+p-2}}0^{d_{k+p-1}}(k\!+\!p\!-\!1)    & \psi(k\!+\!p\!-\!3)   & \psi^2(k\!+\!p\!-\!3)\\
(k\!+\!p\!-\!2) & 0^{d_{k+p-1}}(k\!+\!p\!-\!1)  & (0^m1)^{d_{k+p-1}}0^{d_{k+p}}k                            & \psi(k\!+\!p\!-\!2)   & \psi^2(k\!+\!p\!-\!2)\\
(k\!+\!p\!-\!1) & 0^{d_{k+p}}k              & (0^m1)^{d_{k+p}}\varphi(k)                                & \psi(k\!+\!p\!-\!1)   & \psi^2(k\!+\!p\!-\!1)
\end{array}
\end{equation}
where $\varphi(k)=0^{d_{k+1}}k$ if $p=1$ and $\varphi(k)=0^{d_{k+1}}(k+1)$ otherwise.
As before, we derive that $\varphi^2(i)(0^m1)^m=(0^m1)^m\psi^2(i)$, and therefore $\varphi^2(i)$ has
$(0^m1)^m$ as its prefix for all $i\in\{0,\ldots,k+p-1\}$.
All but the last two rows in \eqref{eq_tabulka_nonsimple_Parry}
then imply  $d_3=\cdots=d_{k+p-2}=m$.

If $k\geq 2$, then also $d_{k+p-1}=d_{k+p}=m$ and we have a contradiction,
$\DBUS=m^\omega$. If, on the other hand, $k=1$, the last line
of~\eqref{eq_tabulka_nonsimple_Parry} then implies
\[\varphi^2(p)=(0^m1)^{d_{p+1}}0^m2\] and necessarily $d_{p+1}=m$,
hence $\DBUS=m(mm\cdots md_pm)^\omega$. Either $d_p=m$, which gives $\DBUS=m^\omega$, or $d_p<m$ and we get a
contradiction with minimality of $k,p\geq 1$.
\end{proof}

Now we can proceed with the proof of the implication $(2)\Rightarrow(1)$ of the main Theorem~\ref{thm_hlavni_superteorem}. 
%

\begin{prop}
Let $\varphi^2\sim\psi^2$. Then $\beta$ is a zero of $x^d-mx^{d-1}-\cdots -mx -n$, $m\geq n\geq1$, $d\geq 1$, with $m=n$ if $d$ is even.
\end{prop}

\begin{proof}
Thanks to
Proposition~\ref{prop_konjugace_implikuje_simpleparry}, relation~\eqref{eq:mm}, and results
in~\cite{MV13}, we can consider only simple Parry numbers $\beta>1$
with \[\DBUS=[mmd_3\cdots d_{k-1}(d_k-1)]^\omega\,,\ k\geq 3\,,\]
where $d_i\in\{0,\ldots,m\}$ and $d_k\neq 0$. Clearly, if $k=3$,
$\DBUS$ is implies that $\beta$ is in the desired class of numbers.

Let $k\geq 4$. Then there are at least four distinct gaps
$\Delta_{0,1,2,3}$ and from~\eqref{eq:substitucesimpleParry}
and~\eqref{thm_antimorfismy} we get
\[
\begin{array}{c|ll|ll} i &
\varphi(i) & \varphi^2(i) & \psi(i) & \psi^2(i)\\ \hline
0           & 0^m1                  & (0^m1)^m0^m2                          & 0^m1      & 0^m2(0^m1)^m\\
1           & 0^m2                  & (0^m1)^m0^{d_3}3                      & 0^m2      & \psi(2)(0^m1)^m\\
\vdots  & \vdots                    & \vdots                                        & \vdots        & \vdots\\
k-3     & 0^{d_{k-2}}(k-2)  & (0^m1)^{d_{k-2}}0^{d_{k-1}}(k-1)  & \psi(k-3) & \psi^2(k-3)\\
k-2     & 0^{d_{k-1}}(k-1)  & (0^m1)^{d_{k-1}}0^{d_k}               & \psi(k-2) & \psi^2(k-2)\\
k-1     & 0^{d_{k}}             & (0^m1)^{d_{k}}                            & \psi(k-1) & \psi^2(k-1)
\end{array}
\]
Using again Proposition~\ref{p:lothaire}, similarly
as before, we obtain that $\varphi^2(i)(0^m1)^m=(0^m1)^m\psi^2(i)$, which implies that
$\varphi^2(i)$ has $(0^m1)^m$ as its prefix for all
$i\in\{0,\ldots,k-2\}$. It directly follows that $\DBUS=[mm\cdots
m(d_k-1)]^\omega$ and Example~\ref{ex_peknatrida} excludes the
case with $k$ even and $d_k<m$.
\end{proof}

\section{Proof $(3)\Rightarrow(1)$ of Theorem~\ref{thm_hlavni_superteorem}}\label{sec_tomzuzana}

In this section we will prove that property $X(-\beta)=\ZMB$ for $\beta>1$ can be satisfied only if $\beta$
is a zero of the polynomial $x^d-mx^{d-1}-\cdots - mx-n$, where $d\geq 1$, $m\geq n\geq 1$, and $d$ is odd or $m=n$.

\begin{lem}\label{lem:transformace}
  Let $d_{-\beta}(\ell)=l_1l_2\dots l_j\dots.$ Then
  $$
  (-\beta)^{j+1}+l_1(-\beta)^j+(l_2-l_1)(-\beta)^{j-1}+
  \dots+(l_j-l_{j-1})(-\beta)-l_j\in[-\beta,1)\,.
  $$
\end{lem}

\begin{proof}
    By the definition of the transformation and the expansion $d_{-\beta}(\ell)$ we have
    $$
    T^j(\ell)=  (-\beta)^j\ell-(-\beta)^{j-1}l_1-\cdots-(-\beta)l_{j-1}-l_j \in [\ell,\ell+1)\,.
    $$
We obtain the statement by multiplying both sides by $(\beta+1)$.
\end{proof}

The above lemma gives us useful estimations on $\beta$ when
a prefix of $d_{-\beta}(\ell)$ is known. In particular, we will make use
of the following implications, valid for $k\geq 1$, $a,b,c\in\{0,1,\dots,m\}$, $a<m$, $b,c>0$ and $t\geq 1$.

\begin{alignat}{3}\label{c1}
\dl&=(m\,0)^k\cdots&\ \Rightarrow&-\beta^{2k+1}+m\beta^{2k}+\cdots+ m\beta+m < m+1\\ \label{c2}
\dl&=(m\,0)^ka\cdots&\ \Rightarrow&-\beta^{2k+2}+m\beta^{2k+1}+\cdots+m\beta^2+a\beta+a>-1\\ \label{c3}
\dl&=(m\,0)^kmb\cdots&\ \Rightarrow& -\beta^{2k+2}+m\beta^{2k+1}+\cdots+m\beta+m-b\geq\frac b\beta-1\\ \label{c4}
\dl&=(m\,0)^k0^{2t-1}c\cdots&\ \Rightarrow& \ \beta^{2k}-m\beta^{2k-1}-\cdots-m\beta-m>-\frac {c}{\beta^{2t}}-\frac {c+1}{\beta^{2t+1}}
\end{alignat}

Assume that $\beta$ satisfies $X(-\beta)=\ZMB$. By
Proposition~\ref{thm_erdos_set}, the gaps in $\Zmb$ are
$\leq 1$ and it follows from $\Delta'_1=\beta-\min(1)<1$ that $\min(1)=m$, hence the string $m0^\omega$ is \hbox{$(-\beta)$-admissible}.
Since the string $(m0)^\omega$ is never admissible (otherwise $\DMBL=(m0)^\omega$ which is impossible), it makes sense to speak about
the greatest index $k\geq 1$ such that $(m0)^k0^\omega$ is admissible.
The following statement provides a necessary condition on the expansion $d_{-\beta}(\ell)$ when $X(\beta)=\ZMB$.

\begin{prop}\label{prop:dva_rozvoje}
Assume that $\beta$ satisfies $X(-\beta)=\ZMB$. If $k\geq 1$ is maximal such that $(m0)^{k}0^\omega$ is admissible, then
    $$
    d_{-\beta}(\ell)=(m0)^{k}ab\cdots,\ ab\neq m0\,.
    $$
\end{prop}

\begin{proof}
From the Ito-Sadahiro admissibility
condition~\eqref{eq_admissibilita_is}, it is clear that the string
$(m\,0)^k$ is a prefix of $d_{-\beta}(\ell)$. When also
$(m\,0)^{k+1}$ is a prefix of $d_{-\beta}(\ell)$ then one can derive from~\eqref{eq_admissibilita_is} that the next
nonzero digit must be on an even position. Thus we necessarily have
    $$
    d_{-\beta}(\ell)=(m0)^{k}ab\cdots,\ ab\neq m0\quad\text{ or }\quad d_{-\beta}(\ell)=(m0)^{k+1}0^{2t-1}c\cdots,\ c,t\geq1.
    $$
Let us exclude the latter case of $d_{-\beta}(\ell)$. Assume that $d_{-\beta}(\ell)=(m0)^{k+1}0^{2t-1}c\cdots$, for some $k,c,t\geq1$. We will prove that $X(-\beta)\neq \Zmb$ by showing that one of the gaps in $\ZMB$ is $>1$, which, by Proposition~\ref{thm_erdos_set}, is impossible in $X(-\beta)$.

Consider extremal strings
    $$
    \begin{array}{cccccc}
      \max(2k+2)&=&0&(m&0)^{k}&(m-1)\\
      \min(2k+2)&=&m&(0&m)^{k}&1
    \end{array}
    $$
    Therefore the value of $\Delta'_{2k+2}$ is
    $$
    \Delta'_{2k+2}=\Big|(-\beta)^{2k+2}+\gamma\big(\min(2k+2)\big)-
    \gamma\big(\max(2k+2)\big)\Big|=\beta^{2k+2}-m\beta^{2k+1}-\cdots -m\beta-m+2
    $$
    and can be estimated using \eqref{c4} by
    $$
    \Delta'_{2k+2}\geq2-\frac c{\beta^{2t}}-\frac{c+1}{\beta^{2t+1}}\geq
    2-\frac c{\beta^2}-\frac {c+1}{\beta^3}\,.
    $$
    If $\beta$ is greater than the Tribonacci constant $\beta_0=1.83929\cdots$, zero of $x^3-x^2-x-1$, then we use $c<\beta$ to conclude that
    $$
    \Delta'_{2k+2}\geq 2-\frac c{\beta^2}-\frac {c+1}{\beta^3} \geq 2-\frac \beta{\beta^2}-\frac{\beta+1}{\beta^3}= 2-\frac1{\beta}-\frac1{\beta^2}-\frac1{\beta^3}>1.
    $$
    If, on the other hand, $\tau=\frac12(1+\sqrt5)\leq\beta\leq \beta_0<2$, then $c=1$, and we have
    $$
    \Delta'_{2k+2}\geq 2-\frac c{\beta^2}-\frac{c+1}{\beta^3}\geq 2-\frac1{\tau^2}-\frac2{\tau^3}= 1+\frac1{\tau^4}>1.
    $$
This concludes the proof, since for $\beta<\tau$, we have $\ZMB=\{0\}\neq X(-\beta)$.
\end{proof}

\begin{pozn}\label{pozn}
In the proof of the above proposition, we have shown that if $\DMBL=(m0)^{j}0^{2t-1}c\cdots$, then $\Delta'_{2j}>1$, and therefore by Proposition~\ref{thm_erdos_set},  $\ZMB\neq X(-\beta)$. This is shown for $j=k+1\geq 2$, but in fact, the argument works for $j=1$, as well.
\end{pozn}

\begin{lem}\label{lem:kratsi}
 Let $d_{-\beta}(\ell)$ have prefix $(m0)^kab$, $k\geq 1$, $ab\neq m0$, and let $\ZMB\neq \{0\}$. Let $z$ be a real number with a $(-\beta)$-representation $z=1(m\,0)^km\bullet 0^\omega$. Then the most significant digit in the $(-\beta)$-expansion $\langle z\rangle_{-\beta}$ of $z$ is at the
position of $(-\beta)^n$ where $n\leq 2k-1$.
\end{lem}

\begin{proof}
By Lemma~3 from~\cite{MaPe}, the statement is true, if we show that
$$
0>z=-\beta^{2k+1}+m\beta^{2k}+m\beta^{2k-2}+\cdots+m\beta^2+m>\frac{-\beta^{2k+1}}{\beta+1}
$$
We prove $\frac{\beta+1}\beta z<0$ for the general case $\dl=(m\,0)^k\cdots.$
  We have
  $$
  \frac{\beta+1}\beta z=-\beta^{2k}\underbrace{-\beta^{2k+1}+m\beta^{2k}+\dots +m\beta +m}_{<1+m, \text{ see }\eqref{c1} } +\frac m\beta<-\beta^2+1+m+\frac m\beta=H\,.
  $$
For $\beta>2$, we use $m<\beta$ to further estimate $H< -\beta^2+1+\beta+1=2+\beta-\beta^2<0$.
 For $\beta\in(1,2)$ we have that $m=\lfloor\beta\rfloor=1$, which gives
  $$
H= -\beta^2+2+\frac1\beta=-(\beta^2-\beta-1)\frac{\beta+1}{\beta}.
  $$
  The last expression is non-positive for $\beta\geq\tau$, i.e.\ such that $\Z_{-\beta}$ is nontrivial.

  \bigskip
  We will verify the second inequality, namely $(\beta+1)z>-\beta^{2k+1}$, separately in two cases, dependently on the form of $d_{-\beta}(\ell)$.
  \begin{enumerate}
        \item
        For  $d_{-\beta}(\ell)=(m\,0)^ka\cdots,\ a<m$ we have
          $$
          \begin{aligned}
            (\beta+1)z&=-\beta^{2k+1}\underbrace{-\beta^{2k+2}+m\beta^{2k+1}+\dots+m\beta^2+a\beta+a}_{>-1,\text{ see }\eqref{c2}}
            +m\beta+m-a\beta-a>\\
            &>-\beta^{2k+1}+(m-a)\beta+m-(a+1)>-\beta^{2k+1}\,.
          \end{aligned}
          $$
        \item
        In case that $\dl=(m\,0)^kmb\cdots$, with $b>0$ we have
        $$
        \begin{aligned}
        (\beta+1)z &=-\beta^{2k+1}\underbrace{-\beta^{2k+2}+m\beta^{2k+1}+\dots+m\beta+m-b}_{>\frac{b}{\beta}-1,\text{ see }\eqref{c3}}+b>\\
        &>-\beta^{2k+1}+\frac b\beta-1+b>-\beta^{2k+1}\,.
        \end{aligned}
        $$

  \end{enumerate}
\end{proof}

Now we are in a state to prove the remaining implication $(3)\Rightarrow(1)$ of Theorem~\ref{thm_hlavni_superteorem}.

\begin{prop}
Let $\beta>1$. If $X(-\beta)=\ZMB$ then $\beta$ is a zero of
$$
x^d-mx^{d-1}-\cdots-mx-n,
$$
where $n=m$ for $d$ even and $n\leq m$ otherwise.
\end{prop}

\begin{proof}
According to Proposition~\ref{prop:dva_rozvoje}, if $X(-\beta)=\ZMB$, then $d_{-\beta}(\ell)=(m\,0)^kab0^\omega,\ ab\neq m0$, where
$k\geq 1$ is such that $(m\,0)^{k+1}0^\omega$ is forbidden while $(m\,0)^k0^\omega$
is admissible.
Consider the number $z$ represented by the forbidden string $1(m\,0)^km\bullet0^\omega$.
According to Lemma~\ref{lem:kratsi}, the most
significant digit of the $(-\beta)$-expansion $\langle z\rangle_{-\beta}$ of $z$ is at the
position  $n\leq{2k-1}$. Since $X(-\beta)=\ZMB$, $z$ is a $(-\beta)$-integer, denote its $(-\beta)$-expansion by
$\langle z\rangle_{-\beta}=\hbox{$z_{2k-1}z_{2k-2}\cdots z_0\bullet 0^\omega$}$.
We have
$$
\begin{aligned}
0&=\gamma\big(1(m\,0)^km\big)-\gamma\big(z_{2k-1}z_{2k-2}\cdots z_0\big)=\\
&=-\beta^{2k+1}+m\beta^{2k}+z_{2k-1}\beta^{2k-1}+(m-z_{2k-2})\beta^{2k-2}+\cdots+z_1\beta+m-z_0\,.
\end{aligned}
$$
Denote by $j$ the maximal index at which the string $z_{2k-1}z_{2k-2}\cdots z_0$ differs from the string $(m0)^k$, i.e.
$z_j\leq m-1$ if $j$ is odd and $z_j\geq 1$ if $j$ is even. We estimate
\begin{equation}\label{eq:blabla}
\begin{aligned}
0&=\gamma\big(1(m\,0)^km\big)-\gamma\big(z_{2k-1}z_{2k-2}\cdots z_0\big)\leq\\
&\leq\underbrace{-\beta^{2k+1}+m\beta^{2k}+\cdots+(m-1)\beta^j+\cdots+m\beta+m}_{<m+1-\beta^j}-(m-z_1)\beta-z_0<\\
&\leq m+1 - \beta^j - (m-z_1)\beta-z_0\,,
\end{aligned}
\end{equation}
where we have used~\eqref{c1}. If $j\geq 2$, we further estimate
\begin{equation}\label{eq:bla}
m+1 - \beta^j - (m-z_1)\beta-z_0 < \beta + 1 - \beta^2 \leq 0\,,
\end{equation}
where the last inequality holds since $\beta$ is necessarily $\geq \tau=\frac12(1+\sqrt5)$. However,
inequalities~\eqref{eq:blabla} and~\eqref{eq:bla} are in contradiction. Therefore $j\in\{0,1\}$, i.e.
the expansion of $z$ can differ from the string $(m\,0)^k$ only at the positions of $(-\beta)^1$ and $(-\beta)^0$.
We therefore have
$$
\begin{aligned}
0&=\gamma\big(1(m\,0)^km\big)-\gamma\big(z_{2k-1}z_{2k-2}\cdots z_0\big)=\gamma\big(1(m\,0)^km\big)-\gamma\big((m0)^{k-1}z_1z_0\big)\\
&=\underbrace{-\beta^{2k+1}+m\beta^{2k}+\dots+m\beta^2+m\beta+m}_{<m+1} - (m-z_1)\beta-z_0<\\
&<m+1- (m-z_1)\beta-z_0\,.
\end{aligned}
$$
It is easy to verify that $0<m+1- (m-z_1)\beta-z_0$ cannot be satisfied if $z_1\leq m-2$. Further, we derive that $z_1=m-1$ implies $z_0=0$.
So the possible pairs $(z_1,z_0)$ are $(m,n)$, $0\leq n\leq m$, and $(m-1,0)$.
Equation $0=\gamma\big(1(m\,0)^km\big)-\gamma\big((m0)^{k-1}z_1z_0\big)$ for these pairs implies that $\beta$ is a zero of the polynomial
$$
\begin{array}{ll}
x^{2k+1}-mx^{2k}-\cdots -mx -n &\text{when }(z_1,z_0)=(m,m-n),\ 0\leq n<m\,,\\
x^{2k}-mx^{2k-1}-\cdots -mx -m &\text{when }(z_1,z_0)=(m,m)\,,\\
x^{2k+1}-mx^{2k}-\cdots -(m-1)x -m&\text{when }(z_1,z_0)=(m-1,0)\,.
\end{array}
$$
In order to conclude the proof, we have to show that the latter case does not occur. In fact, the irreducible polynomial  $x^{2k+1}-mx^{2k}-\cdots -(m-1)x -m$ has a unique zero $\beta>1$, for which
$\dl=[(m\,0)^k(0\,m)^k]^\omega$. However, by Remark~\ref{pozn}, $d_{-\beta}(\ell)=(m0)^{k}0^{2t-1}c\cdots$, $c,t\geq1$, implies $X(-\beta)\neq\ZMB$.
\end{proof}


\section{Comments and examples}\label{sec_5examples}

\begin{description}

\item[Confluent Parry numbers of an even degree.]

Let $\beta$ be a confluent Parry number of an even degree $d\geq 2$ with
minimal polynomial $p(x)=x^d-mx^{d-1}-\cdots-mx-n$, $m>n\geq 1$.
We already know from Example~\ref{ex_peknatrida} that the sets
of distances in $\ZBP$ and $\ZMB$ do not coincide, hence
$\varphi^2\nsim\psi^2$. However, one can use similar approach as in~\cite{MV13}
to show that cutting every distance $\Delta'_{d-1}>1$ in $\ZMB$ into
$\Delta'_{d-1}=1+\frac{n}{\beta}=\Delta_0+\Delta_{d-1}$, we obtain a structure that can be coded 
by an infinite word with the same language as the language of $u_\beta$. Formally, 
we apply on $u_{-\beta}$ a morphism $\pi:\{0,\ldots,d-1\}^*\rightarrow\{0,\ldots,d-1\}^*$,
\[
\pi(i)=\begin{cases}
i & \text{if $i\in\{0,\ldots,d-2\}$,}\\
0(d-1) & \text{if $i=d-1$.}
\end{cases}
\]
Then it can be verified that the words $u_\beta$ and
$\pi(u_{-\beta})$ have the same language. Indeed, one can show that $\pi(u_{-\beta})$ is a fixed point of
a morphism $\widetilde\psi$, which is the unique morphism for which
$\pi\circ\psi^2=\widetilde\psi\circ\pi$, and $\varphi^2\sim\widetilde\psi$.

\item[Quadratic and cubic numbers.]

The comparison of $(\pm\beta)$-integers for quadratic numbers $\beta$ was done in~\cite{MV13}.
It was shown that the sets of distances in $\ZBP$ and in $\ZMB$ coincide if and only if $\varphi^2\sim\psi^2$, which is equivalent to 
$\ZMB=X(-\beta)$. This happens precisely for the class of zeros $\beta>1$ of $x^2-mx-m$, $m\geq 1$.
For other quadratic Parry numbers, $\Zmb$ always contains distances $>1$.


Considering cubic Parry numbers, we see a difference. There may be bases $\beta$, for which the sets of distances in $\ZBP$ and $\ZMB$
coincide, however, $\varphi^2\not\sim\psi^2$. We may study the possible situations starting with the classification of 
cubic Pisot numbers, which can be derived from~\cite{Aki00} and~\cite{Bas02}.
For illustration, we restrict ourselves to cubic Pisot units, i.e. $\beta>1$ with minimal polynomial
$p(x)=x^3-ax^2-bx-c$, $c=\pm 1$.

\begin{itemize}
\item
If $p(x)=x^3-mx^2-mx-1=0$, $m\geq 1$, then $\beta$ is a confluent Parry number for which $\ZMB=X(-\beta)$, and from 
$$
d_\beta(1)=mm10^\omega\,,\quad d_{-\beta}(\ell)=m01^\omega\,,
$$
we may derive for the sets of distances that $\{\Delta_1,\Delta_2,\Delta_3\}=\{\Delta'_1,\Delta'_2,\Delta'_3\}$
and that $\varphi^2\sim\psi^2$.

\item
If $p(x)=x^3-(m+1)x^2+x-1=0$, $m\geq 1$, or  $p(x)=x^3-(m+1)x^2+1=0$, $m\geq 2$, then 
$\{\Delta_1,\Delta_2,\Delta_3\}=\{\Delta'_1,\Delta'_2,\Delta'_3\}$, however, the morphisms are not conjugated.

\item
For all the other cases of cubic Pisot units, we derive from $d_\beta(1)$ and $d_{-\beta}(\ell)$ that $\ZMB$ contains a distance $>1$.

\end{itemize}
%


\item[More about spectra of Pisot numbers.]

The set $X(\beta)$  is a special case of a more gene\-ral notion of a spectrum of a real number $\beta>1$, defined as the set of $p(\beta)$ where $p$ ranges over all polynomials with coefficients restricted
to a finite set of integers. In particular,
$$
X^{r}(\beta) = \Big\{\sum_{j=0}^N a_j \beta^j : N\in\N,\, a_j\in\{0,1,\dots,r\}\Big\}\,.
$$
For an extensive overview of the problem of spectra, see for example~\cite{AkiKomo}. 
A general result by Feng and
Wen~\cite{FengWen02} states that for a Pisot number $\beta$ and
$r+1>\beta$, the sequence of distances in
$X^{r}(\beta)$ can be generated by a substitution. However,
neither an explicit prescription for the substitution, nor the values of
distances and their frequencies are known in general.
Bugeaud in 2002~\cite{Bug02} gives these in case that $\beta$ is a
multinacci number, i.e. zero of~\eqref{eq:confluentpos} for
$m=n=1$, and $r=\lfloor\beta\rfloor=1$. Garth and
Hare in 2006~\cite{GarthHare} provide the substitution for any zero
of~\eqref{eq:confluentpos} and $r=\lfloor\beta\rfloor$. Notably, in both
cases, the substitution can be simply obtained by observing that
$X^{\lfloor\beta\rfloor}(\beta)=\Z_\beta^+$ and using the
canonical substitution for $\beta$-integers given in 1995 by
Fabre~\cite{Fab95}.

\item[Arithmetics.]

Among the desired properties of a number system is that the set of numbers with finite expansions is closed under arithmetical operations, in particular, under addition and subtraction. For bases $\beta>1$, this the so-called finiteness property rewrites as ${\rm Fin}(\beta)=\Z[\beta,\beta^{-1}]$. It was shown in~\cite{FrSo} that the finiteness property holds for all confluent Parry numbers.
An analogous finiteness property for negative base system has not yet been sufficiently explored. So far, the only known class of numbers $\beta>1$ such that ${\rm Fin}(-\beta)=\Z[\beta,\beta^{-1}]$ were the zeros of $x^2-mx+n$, $m-2\geq n\geq 1$, as shown in~\cite{MPV11}.

One could have expected that the bases for which $\Z_{-\beta}=X(-\beta)$ are proper candidates for the negative finiteness property.  However, for almost all confluent Parry numbers, this is not valid. To see this, it is sufficient to consider the following examples
of sum of numbers in $\Zmb$ with infinite expansions:
\begin{itemize}
\item For $\beta$ zero of $x^{2k}-mx^{2k-1}-\cdots-mx-m$, $m\geq 1$, one has
$$
\langle m+1\rangle_{-\beta}=1m0\bullet0^{2k-3}11m\big[0^{2k-3}110\big]^\omega\,;
$$
\item for $\beta$ zero of $x^3-mx^2-mx-n,\ m>n\geq 1$, one has
$$
\langle -\beta+m+1\rangle = 0\bullet0(m-n+1)(m-n+1)1(n+1)^\omega\,;
$$
\item for $\beta$ zero of $x^{2k+1}-mx^{2k}-\cdots-mx-n,\ m>n\geq 1,\ k\geq 2$, one has
$$
\langle -\beta+m+1\rangle_{-\beta}=0\bullet0^{2k-1}(m-n+1)(m-n+1)0\big[0^{2k-3}1(n+1)n\big]^\omega.
$$
\end{itemize}
On the other hand, in case when $\beta$ is of odd degree and $n=m$, we conjecture that the finiteness property is satisfied. The proof for the smallest such $\beta$, namely the Tribonacci constant, zero of $x^3-x^2-x-1$, obtained recently by one of the authors, can be found in~\cite{Va13}.

\item[Measurably isomorphic transformations.]

An interesting comparison of $(\pm\beta)$-nume\-ration is done by Kalle in~\cite{Kal13}, where the similarity of the transformations $T_\beta$ and $T_{-\beta}$ is studied. It is shown that although the transformations cannot be isomorphic (since they have different number of fixed points), one can sometimes find a measurable isomorphism between \hbox{$(\pm\beta)$-transformations}. Among all bases $\beta\in(1,2)$, this happens precisely if $\beta$ is a multinacci number.

Note that for $\beta\in(1,2)$ the notions of multinacci numbers and confluent Parry numbers coincide. Hence for bases $\beta<2$ we get that the existence of a measurable isomorphism between $(\pm\beta)$-transformations is equivalent to all conditions in~Theorem~\ref{thm_hlavni_superteorem}.
One could expect this measurable isomorphism property to be related to properties in~Theorem~\ref{thm_hlavni_superteorem} also for $\beta>2$. Nevertheless, it is conjectured in~\cite{Kal13} that among all $\beta>1$, the measurable isomorphism property will hold exactly for the zeros of polynomials of the form $x^d-mx^{d-1}-\ldots -mx-n$, where $m\geq n\geq 1$ and $d\geq 1$ is arbitrary. As we exclude in our results the subclass of $\beta$'s with minimal polynomial of even degree $d$ with $m>n$, it seems that properties in~Theorem~\ref{thm_hlavni_superteorem} are, in general, not equivalent to $T_\beta$ and $T_{-\beta}$ being measurably isomorphic.

\end{description}

\section*{Acknowledgements}
This work was supported by the Czech Science Foundation, grant No.\ 13-03538S. We also acknowledge financial support of
the Grant Agency of the Czech Technical University in Prague, grant No.\ SGS11/162/OHK4/3T/14.


\end{document}